\documentclass{amsart}

\usepackage{amsmath,amssymb,amsthm,euro} 
\usepackage{bbm}
\usepackage{stmaryrd}

\newcommand\F{\mathcal F}
\newcommand\I{\mathbbm{1}}
\renewcommand{\P}{\mathbb{P}}
\newcommand{\M}{\mathcal M}
\renewcommand{\S}{\mathcal T}
\newcommand{\D}{\mathcal D}
\newcommand{\SI}{\mathcal {SI}}
\renewcommand{\rho}{\varrho}

\newcommand\ma{\mathcal{A}}

\renewcommand\i{\infty}
\newcommand\ve{\varepsilon}
\newcommand\eps{\varepsilon}

\newcommand\T{0\le t\le T}
\newcommand\1{0\le t\le 1}

\newcommand\p{\mathbb{P}}
\newcommand\N{\mathbb{N}}
\newcommand\E{\mathbb{E}}
\newcommand\R{\mathcal{R}}

\newcommand\RR{\mathbb{R}}

\newcommand\LI{{L^\i (\P)}}
\newcommand\Lzero{{L^0 (\P)}}

\newcommand\Ltwo{{L^2 (\P)}}
\newcommand\Linfty{{L^\infty (\P)}}

\newcommand\LIOT{\infty}%{{L^\i (\Omega\times [0,T])}}
\newcommand\LIOO{\infty}%{{L^\i (\Omega\times [0,1])}}

\newtheorem{theorem}{Theorem}[section]

\newtheorem{definition}[theorem]{Definition}
\newtheorem{lemma}[theorem]{Lemma}
\newtheorem{proposition}[theorem]{Proposition}

\theoremstyle{definition}

\newtheorem{remark}[theorem]{Remark}

\title[A direct Proof of the Bichteler--Dellacherie Theorem]{A direct Proof of the Bichteler--Dellacherie Theorem and connections to Arbitrage}

\begin{document}
\author { Mathias Beiglb\"ock, Walter Schachermayer, Bezirgen Veliyev}

\address{Fakult\"at f\"ur Mathematik, Universit\"at Wien\endgraf
Nordbergstra\ss e 15\\ 1090 Wien, Austria\endgraf\ }
\email{\endgraf mathias.beiglboeck@univie.ac.at,  \endgraf 
walter.schachermayer@univie.ac.at,\endgraf 
bezirgen.veliyev@univie.ac.at. \endgraf}
%\email{bezirgen.veliyev@univie.ac.at}
%\email{walter.schachermayer@univie.ac.at}
\thanks{The first author gratefully acknowledges financial support from the Austrian Science
Fund (FWF) under grant  P21209. The second author gratefully acknowledges financial support from the Austrian Science Fund (FWF) under grant P19456, from
the Vienna Science and Technology Fund (WWTF) under grant MA13, from the Christian Doppler Research Association, and from the ERC Advanced Grant.
 The third author gratefully acknowledges financial support from the Austrian Science Fund (FWF) under grant P19456.}


\maketitle

\begin{abstract}
We give an elementary proof of the celebrated Bichteler-Dellacherie Theorem which states that the class of stochastic processes $S$ allowing for a useful integration theory consists precisely of those processes which can be written in the form $S=M+A$, where $M$ is a local martingale and $A$ is a finite variation process. In other words, $S$ is a good integrator if and only if it is a semi-martingale.

We obtain this decomposition rather directly from an elementary discrete-time Doob-Meyer decomposition. By passing to convex combinations we obtain a direct construction of the continuous time decomposition, which then yields the desired decomposition.

As a by-product of our proof we obtain a characterization of semi-martingales in terms of a  variant of \emph{no free lunch}, thus extending a result from \cite{DeSc94}.
\end{abstract}

\section{Introduction}
%We start by explaining some concepts and notation.
% Fixing the filtered probability space $(\Omega, \F, (\F_t)_{\T}, \p)$   we call a {\it simple integrand} a  stochastic process $H=(H_t)_{\T}$ of the form 

We fix filtered probability space $(\Omega, \F, (\F_t)_{\T}, \p)$ satisfying the usual conditions.
A {\it simple integrand} is a 
stochastic process $H=(H_t)_{\T}$ of the form 
\begin{align}\label{W2}
H_t = \sum^n_{j=1} f_j \I_{
\rrbracket 
\tau_{j-1} ,\tau_j 
\rrbracket
} (t), \qquad \T,
\end{align}
where $n$ is a finite number, $0=\tau_0 \le \tau_1\le \cdots \le \tau_n =T$ is an increasing sequence of stopping times, 
and $f_j\in L^\i (\Omega, \F_{\tau_{j-1}}, \p)$.%(compare, e.g., \cite[page 51]{Prot04}  or \cite[page 12]{RoWi00b})

Denote by $\SI =\SI (\Omega, \F, (\F_t)_{\T}, \p)$ the vector space of (equivalence
classes of) simple integrands.

For every bounded, $\F\otimes \mathfrak{B}_{[0,T]}$-measurable function $G: \Omega\times [0,T]\to \RR$ we define
$$\|G\|_\LIOT=\sup_{0\leq t\leq T} \big\| G_t\big\|_\LI$$
so that $\|.\|_\LIOT$ is a norm on $\SI$.
%\medskip
Given a (c\`adl\`ag, adapted) stochastic process $S=(S_t)_{0\leq t\leq T}$
we may well-define the integration operator $I_S:\SI \to L^0 (\Omega, \F,\p)$,
\begin{equation}\label{W3}
I_S\Big(
\sum^n_{j=1} f_j \I_{
\rrbracket 
\tau_{j-1} ,\tau_j 
\rrbracket
} \Big)= \sum^n_{j=1} f_j (S_{\tau_j}-S_{\tau_{j-1}}) =:(H\cdot S)_T.
\end{equation}
Note that only a finite Riemann sum is involved in this definition of an integral, so that we do not (yet) encounter any subtleties of limiting procedures.

However, if we seek to extend this operator to a larger class of integrands by approximation with simple integrands, we have to demand that the operator $I_S$ enjoys some minimal continuity properties. %satisfies some minimal continuity assumptions. 
% To deserve the name integral, the operator $I$ should enjoy some minimal 
A particularly weak requirement is that uniform convergence of a sequence of simple integrands $H^n$ should imply convergence of the integrals $I_S(H^n)$ in probability.  
\begin{definition}
(see, e.g., \cite[page 52]{Prot04}, \cite[page 24]{RoWi00b})
A real-valued, c\`{a}dl\`{a}g, adapted process $S=(S_t)_{\T}$ is called a {\sffamily good integrator} if the integration operator %$I$ defined in \eqref{W3}
$$I_S:\SI \to L^0 (\Omega, \F, \p)$$
is continuous, if we equip $\SI$ with the norm topology induced by $\|\cdot\|_\LIOT$, and $L^0(\Omega,\F,\p)$ with the topology of convergence in probability,
respectively.
\end{definition}
If $S$ is a good integrator, it is possible to extend the operator $I_S$ to the space of all bounded adapted c\`agl\`ad  processes without major technical difficulties (\cite[Capter 2]{Prot04}). 
%The definition is designed in such a way that, by continuity, we can extend the integration operator $I$ to the $\|\cdot \|_\i$-completion of the normed space
%$\SI$, which is the Banach space of a.s.\ uniformly bounded, predictable processes $H=(H_t)_{\T}$. 

In other words, the above notion ensures, {\it essentially by definition},
that the procedures involved in extending the integration \eqref{W3} from finite Riemann sums to their appropriate limits, work out properly for a 
good integrator $S$. But of course, the above definition of a good integrator is purely formal, and simply translates the delicacy of the well-definedness of an integral (which involves a limiting procedure) into an equivalent condition.

%To include more general, say, bounded predictable integrands is more delicate and requires a structural description of  $S$.

 The achievements of the Strasbourg school of P.~A.~Meyer and the work of
G.~Mokobodzki culminated in the theorem of Bichteler--Dellacherie (\cite[Theorem 43, Chapter 3]{Prot04}, \cite[Theorem 16.4]{RoWi00b}), which provides an explicit and 
practically useful characterization of the set of processes allowing for a powerful stochastic integration theory.

\begin{theorem}\label{th2}
(\cite{Bich79}, \cite{Bich81}, \cite{Dell80}):
For a real-valued, c\`{a}dl\`{a}g, adapted process $S=(S_t)_{\T}$ the following are equivalent:
\begin{enumerate}
\item $S$ is a good integrator.
\item $S$ may be decomposed as $S=M+A$, where $M=(M_t)_{\T}$ is a local martingale and $A=(A_t)_{\T}$ an adapted process of finite variation. 
\end{enumerate}
We then say $S$ is a  {\sffamily semi-martingale}.
\end{theorem}

The implication $(2)\Rightarrow (1)$ is a straightforward verification. Indeed it is rather trivial that  a c\`adl\`ag, adapted process $A$ with a.s.\ paths of finite variation is a good integrator, where the integral may be defined pathwise. As regards the local martingale part $M$, the assertion that $M$ is a good integrator, is an extension of It{\^o}'s fundamental insight (\cite{Ito44, KuWa67}) that an $L^2$-bounded martingale defines an integration operator which is continuous from $(\SI, \|.\|_{\LIOT})$ to $\Ltwo$. 

 The remarkable implication is $(1) \Rightarrow (2)$ which provides an explicit 
characterization of good integrators.

The main aim of this paper is to give a proof of this implication which is inspired by (no) arbitrage-arguments. We note that our argument does not rely on  the continuous time Doob-Meyer decomposition nor any change of measure techniques. Instead, we shall construct the desired representation rather directly from a discrete time Doob-Meyer decomposition.\footnote{Thus, our proof is --- in spirit --- closely related to the proofs of the continuous time Doob-Meyer Theorem for super-martingales given by \cite{Rao69} (see also \cite{IkWa81}, \cite{KaSh91}) and \cite{Bass96}  (see also \cite{Prot04}). %The difference is that these arguments use weak compactness while we shall apply Koml\`os type arguments to pass to appropriate limits.
} 
As an important by-product we also obtain a direct proof of the decomposition of a locally bounded semi-martingale (see Theorem \ref{BMT} below).

\medskip

Let us now enter the realm of Mathematical Finance. 
%Given a simple process $H$, the \emph{value process} $(H\cdot S)$ is defined by $(H\cdot S)_t:=\big((\I_{[0,t]}H)\cdot S\big)_t$ for $0\leq t\leq T$ and  models the total gains or losses from trading according to $H$ in the stock $S$ during the time interval $[0,t]$.

Here $S$ models the (discounted) price process of some ``stock'' $S$, say, a share of company XY. People may trade the stock $S$: at time $t$ they can hold $H_t$  units of the stock $S$. Following a \emph{trading strategy} $H=(H_t)_{0\leq t \leq T}$, which is assumed to be a \emph{predictable} process, the accumulated gains or losses up to time $t$ then are given by the random variable 
\begin{align}\label{gfds}(H\cdot S)_t=\int_{0}^t H_u\, dS_u, \quad 0\leq t\leq T.\end{align} 
The intuition is that during an infinitesimal  interval $[u, u+du]$ the strategy $H$ leads to a random gain/loss $H_u \, dS_u$. In the case when the predictable process $H$ is a step function, i.e.\ if $H$ is a simple integrand, the stochastic integral \eqref{gfds} becomes a finite Riemann sum. Hence in this case it is straightforward to justify this infinitesimal reasoning.

 The dream of an investor is the possibility of an \emph{arbitrage}\footnote{The basic axiom of mathematical finance is that arbitrages do not exist: there is no such thing as a free lunch!}. Roughly speaking, this means the existence of a trading strategy, where you are sure not to lose, but where you possibly may  win. Mathematically speaking -- and leaving aside technicalities -- this translates into the existence of a predictable process $H=(H_t)_{0\leq t \leq T}$ such that the negative part $(H\cdot S)_T^-$ of the gains/losses accumulated  up to the terminal date $T$ is zero, while the positive part $(H\cdot S)_T^+$ is not. We now give a technical variant of this intuitive idea of an arbitrage.
 
%The following definition, which is easily seen to be equivalent to the one given in \cite{DeSc94}, is formulated in such a way that it stresses the  formal analogy to the definition of a semi-martingale, which also involves the $\|.\|_\i$-topology and the topology of convergence in probability.
\begin{definition} \label{def4}
(\cite[section 7]{DeSc94}):
A real-valued, c\`{a}dl\`{a}g, adapted process $S=(S_t)_{\T}$ allows for a {\sffamily free lunch with vanishing risk for simple integrands
} if there is a
sequence $(H^n)^\i_{n=1}$ of simple integrands such that, for $n\to \infty$,
\begin{align}
(H^n\cdot S)_T^+\quad\  &\not \to 0 \quad \mbox{ in probability}. \tag{\emph{FL}}\label{FL}\\
 \sup_{0\leq t\leq T} \big\| (H^n\cdot S)_t^{-} \big\|_\LIOT= \big\| (H^n\cdot S)^{-} \big\|_\LIOT  &\to 0, \tag{\emph{VR}}\label{VR}
\end{align}
Rephrasing the converse, $S$ therefore admits {\sffamily no free lunch with vanishing risk} (NFLVR) for simple integrands if for every sequence $(H^n)^\i_{n=1}\in \SI$ satisfying \eqref{VR} we have
\begin{align}
(H^n\cdot S)_T \to 0 \quad \mbox{ in probability}. \tag{\emph{NFL}}\label{NFL}
\end{align} 
\end{definition}
The Mathematical Finance context allows for the following interpretation:  A {\it free lunch with vanishing risk for simple integrands} 
indicates that $S$ allows for a sequence of trading schemes $(H^n)^\i_{n=1}$, each $H^n$ involving only finitely many rebalancings of the portfolio, such that the losses tend to 
zero in the sense of \eqref{VR}, while the terminal gains \eqref{FL} remain substantial as $n$ goes to infinity.\footnote{The sequence of random variables $((H^n\cdot S)^+_T)_{n=1}^\infty$ does not converge to zero in probability iff there is some $\alpha>0$ such that $\P[(H^n\cdot S)^+_T\geq \alpha]\geq \alpha$, for infinitely many $n\in \N$.
%We give a verbal illustration of a free lunch with vanishing risk, for say $\alpha=\frac1{10}$. Letting $\eps=\frac1{10\ 000\ 000}$ it asserts that an agent can make finitely many transactions such that the losses never exceed $1$ euro, while with probability $\alpha=\tfrac1{10}$ is one million. 
}

It is important to note that the condition \eqref{VR} of \emph{vanishing risk} pertains to the maximal losses of the trading strategy $H^n$ during the entire interval $[0,T]$: if the left hand side of \eqref{VR} equals $\eps_n$ this implies that, with probability one, the strategy $H^n$ \emph{never}, i.e.\ for no $t\in [0,T]$, causes an accumulated loss of more than $\eps_n$. 

Here is the mathematical theorem which gives the precise relation to the notion of semi-martingales.
\begin{theorem}\cite[Theorem 7.2]{DeSc94}\label{th1}
Let $(S_t)_{\T}$ be a real-valued, c\`{a}dl\`{a}g, locally bounded process based on and adapted to a filtered probability space $(\Omega, \F,(\F_t)_{\T}, \p).$
If $S$ satisfies the condition of {\sffamily no free lunch with vanishing risk for simple integrands}, then $S$ is a semi-martingale.
\end{theorem}
In this theorem we only get one implication, as opposed to the characterization of a semi-martingale in the Bichteler--Dellacherie Theorem \ref{th2}.
Indeed, trivial examples show that a semi-martingale $S=(S_t)_{\T}$ does not need to satisfy the condition of ``no free lunch with vanishing risk for simple
integrands''. For example, consider  $S_t=t$ and $H^n_t\equiv H_t\equiv 1$, for $\T$. Then, for each $n\in\N$, we have that $(H^n\cdot S)_T=T$ which certainly
provides a ``free lunch with vanishing risk''. 

The proof of Theorem \ref{th1} which is  given in (\cite[Th 7.2]{DeSc94}) relies on the Bichteler--Dellacherie Theorem. The starting point of the present paper was
the aim to find a proof of Theorem \ref{th1} which does not rely on this theorem. Rather we wanted to use Komlos' lemma and its ramifications which allows
in rather general situations to pass to limits of sequences of functions and/or processes by forming convex combinations.

It came as a pleasant surprise that not only it is possible to prove Theorem \ref{th1} in this way, but that these arguments also yield a constructive proof
of the Bichteler--Dellacherie Theorem which is based on an intuitive and seemingly naive   idea. 

To relate the themes of \ref{th1} and the Bichteler--Dellacherie Theorem \ref{th2} we introduce for the context of this paper the
following definition which combines the two theorems.

\begin{definition}\label{NFLVRLI}
Given a process $S=(S_t)_{\T}$, we say that $S$ allows for a {\sffamily free lunch with vanishing risk and little investment}, if there is a sequence $(H^n)^\i_{n=1}$
of simple integrands as in Definition \ref{def4} above, satisfying \eqref{FL}, \eqref{VR}, and in addition  
\begin{align}
\lim_{n\to \infty} \| H^n \|_\infty = 0. \tag{\emph{LI}}\label{SI}
\end{align}
\end{definition}

The finance interpretation of \eqref{SI} above is that, on top of the requirements of {\it free lunch with vanishing risk}, %, at the $n$'th stage, 
the holdings $H^n_t$ in the stock $S$ is small when $n$ tends to infinity, a.s.\ for all $\T$.
Speaking loosely in more economic terms: $S$ allows for a {\it free lunch with vanishing risk  and little investment} if there are strategies which are almost an arbitrage and which are only involve the holding (or short-selling) of a few stocks. 

We may resume our findings in the following theorem which combines and strengthens the content of  Theorem \ref{th1} and the Bichteler--Dellacherie Theorem \ref{th2}.

\begin{theorem}\label{BMT}
For a locally bounded, real valued, c\`{a}dl\`{a}g, adapted process $S=(S_t)_{\T}$ the following are equivalent. 
\begin{enumerate}
% \item $S$ is a good integrator. 
\item $S$ admits {\sffamily no free lunch with vanishing risk and little investment}, i.e., for any sequence $H^n\in\SI$ with
$\lim_n \big\| (H^n\cdot S)^{-} \big\|_\LIOT=\lim_n \| H^n \|_\LIOT = 0$ we find that $ \lim_n (H^n\cdot S)_T^+=0$ in probability.
\item  $S$ is a semi-martingale in the sense of Theorem \ref{th2} (2).  
\end{enumerate}
%When these two conditions are satisfied, the variation process $A$ may be chosen to be predictable and then the decomposition $S=M+A$ is unique.
\end{theorem}
In the case of general processes $S$, which are not necessarily locally bounded, Theorem \ref{BMT} does not hold true any more. Indeed, \cite[Example 7.5]{DeSc94} provides an adapted c\`{a}dl\`{a}g process $S=(S_t)_{0\leq t \leq 1}$ which is not a semi-martingale and for which every simple process $H\in \SI$ satisfying 
$$ \left\| (H^n\cdot S)^{-} \right\|_{\LIOT} \leq 1$$ 
 is constant. Therefore, $S$ trivially verifies the condition of  {\sffamily no free lunch with vanishing risk} (and in particular {\sffamily no free lunch with vanishing risk and little investment}). 

\medskip

%We say that a sequence of simple integrands $H^n$ has  {\sffamily unlikely risk}\footnote{Silver medal for bad notation?} 
%if \begin{align} & \sup_{0\leq t\leq T} \left( (H^n\cdot S)_t\right)^{-} \to 0 \mbox{ in probability}. \label{UR}\tag{\emph{ur}}
%\end{align}
%\begin{definition}
%We say that a process $S$ allows for a {\sffamily poor man's arbitrage}, in short PMA, if there is a sequence of simple processes $H^n$ satisfying \eqref{FL}, \eqref{SI} and in addition has only an {\sffamily unlikely risk}, i.e.\Êverifies
%\end{definition}
%The analogue of Theorem \ref{BMT} for general processes $S$ now reads as follows.

But by appropriately altering the condition \eqref{VR} above, we can also formulate a theorem which is analogous to Theorem \ref{BMT} and which implies, in particular, the classical theorem of Bichteler-Dellacherie in its general setting.
\begin{theorem}\label{UBMT}
For a real valued, c\`{a}dl\`{a}g, adapted process $S=(S_t)_{\T}$ the following are equivalent.
\begin{enumerate}
\item For any sequence of simple integrands $H^n$ with $\lim_{n} \| H^n \|_\LIOT=0$ and 
\begin{align*}%\label{UR}%\tag{WVR}
\lim_n \sup_{0\leq t\leq T} \left( (H^n\cdot S)_t\right)^{-}=0 \mbox{ in probability}
\end{align*}
 we find that $ lim_n (H^n\cdot S)_T^+=0$ in probability too.
\item $S$ is a semi-martingale in the sense of Theorem \ref{th2} (2).
\end{enumerate}
\end{theorem}

\begin{remark}
We also mention the interesting paper \cite{KaPl10}. In the setting of a non-negative process $S$, it is shown that $S$ is a semi-martingale if and only if it satisfies a weakened NFLVR-condition. Moreover it is pointed out in \cite{KaPl10} what has to be altered to include the case where $S$ is just locally bounded from below. The authors also succeed to avoid the Bichteler-Dellacherie Theorem, which is replaced by the continuous time Doob-Meyer Theorem.
\end{remark}

Finally, we  thank Dirk Becherer and Johannes Muhle-Karbe for useful remarks.

\section{The idea of the proof}

%{\bf Idea of the proof:} 
We consider a c\`{a}dl\`{a}g, real-valued, adapted process $S=(S_t)_{\T}.$ We want to decide whether $S$ is a semi-martingale, 
and whether $S$ allows for a ``free lunch with vanishing risk and little investment''.
We only consider the finite horizon case, where from now on we normalize to $T=1$; the extension to the infinite horizon case is straight-forward
(see \cite{Prot04}, \cite{RoWi00b}). We also assume that $S_0=0$.

We start with the situation when $S$ is locally bounded and shall discuss the general case later.

Noting the fact that being a semi-martingale is a local property, we may and do assume by stopping that $S$ is uniformly bounded, 
say $\|S\|_\i \le 1$ (compare \cite{DeSc94}%\marginpar{precise ref. missing}
).
For $n\in\N$ consider the discrete process $S^n=(S_{\frac{j}{2^n}})^{2^n}_{j=0}$ obtained by sampling the process $S$ at the $n$'th dyadic points. The process $S^n$ may
be uniquely decomposed into its Doob--Meyer components
$$ S^n=M^n+A^n$$
where $(M^n_{\frac{j}{2^n}})^{2^n}_{j=0}$ is a martingale and $(A^n_{\frac{j}{2^n}})^{2^n}_{j=0}$ a predictable process with respect to the filtration 
$(\F_{\frac{j}{2^n}})^{2^n}_{j=0}$: indeed, letting $A^n_0=0$ it suffices to define 
\begin{align}\label{W13}
A^n_{\frac{j}{2^n}} -A^n_{\frac{j-1}{2^n}} &=\E \big[S_{\frac{j}{2^n}}-S_{\frac{j-1}{2^n}} | \F_{\frac{j-1}{2^n}}\big], \qquad &j=1,\ldots , 2^n, 
\\
M^n_j &=S^n_j -A^n_j, \qquad & j=0,\ldots, 2^n.\label{W13a}
\end{align}
Observe that we do not have any integrability problems in \eqref{W13} as $S$ is bounded.

The idea of our proof is -- speaking very roughly and somewhat oversimplifying -- to distinguish two cases.

{\it Case 1: The processes $(M^n)^\i_{n=1}$ and $(A^n)^\i_{n=1}$ remain bounded} (in a sense to be clarified below).
In this case we shall apply Komlos type arguments to pass to limiting processes $M=\lim_{n\to\i}M^n$ and 
$A=\lim_{n\to \i} A^n$ which then will turn out to be a local martingale and a finite variation process 
(in continuous time) with respect to the filtration $(\F_t)_{\T}$. Hence in this case we find that $S$ is a 
semi-martingale in the sense of Theorem \ref{th2} (2).

{\it Case 2: The processes $(M^n)^\i_{n=1}$ and/or $(A^n)^\i_{n=1}$ do not remain bounded.}
In this case we shall construct a sequence of simple integrands $(H^k)^\i_{k=1}=\left( (H^k_t)_{\1}\right) ^\i_{k=1}$ for the process $(S_t)_{\1}$ which yield
a {\it free lunch with vanishing risk and little investment}. Here is some finance intuition why such a construction should be possible: under the assumption
of Case 2 we may find a sequence $\ve_n >0$ tending to zero such that $(\ve_n M^n)^\i_{n=1}$ and $(\ve_n A^n)^\i_{n=1}$ still do not ``remain bounded''. Noting
that $\|\ve_n M^n +\ve_n A^n\|_\i = \|\ve_n S\|_\i \le \ve_n$ we get an unbounded sequence $(\ve_nM^n)^\i_{n=1}$ of local martingales and/or an unbounded sequence
$(-\ve_nA^n)^\i_{n=1}$ of predictable processes which are close to each other in the uniform topology. Oversimplifying things slightly, this may be interpreted that \emph{the predictable 
process $-A^n$ traces closely the martingale $M^n$.} This ability of nearly reproducing the random movements of the martingale $M^n$ by the predictable movements
of the process $A^n$ should enable a smart investor to achieve a free lunch by forming simple integrands $(H^k)^\i_{k=1}$ which can be chosen such that $\lim_{k\to \i}
\|H^k\|_\i =0$.

\medskip

Of course, this is only a very crude motivation for the arguments in the next section, where we have to be more precise what we mean by ``to remain bounded'' (in the sense of quadratic variation or total variation, in the sense of
$L^\i, L^2,$ or $L^0$, etc etc) and where we have to do a lot of stopping and passing to convex combinations to make the above ideas mathematically rigorous. The crucial issue
is that a successful completion of the above program will simultaneously yield proofs for the Bichteler--Dellacherie Theorem (Theorem \ref{th2}) as well as for Theorem \ref{th1}.
Indeed, it will prove Theorems \ref{BMT} and \ref{UBMT} which contain these theorems.

\section{Two preliminary decomposition results.}

In this section we give two auxiliary results which are somewhat technical but already establish the major portion of our proof of the Bichteler-Dellacherie Theorem.

\begin{proposition}\label{DiscreteApp}
Let $S=(S_t)_{\1}$ be a c\`{a}dl\`{a}g, adapted process satisfying $S_0=0$, $\| S\|_\LIOO\leq 1$ and   {\sffamily no free lunch with vanishing risk and little investment}.
Denote by $A^n$ and $M^n$ the discrete time Doob decompositions as in \eqref{W13} resp.\ \eqref{W13a}.

For $\eps>0$, there exist a constant $C>0$ and a sequence of $\{\tfrac j{2^n}\}_{j=1}^{2^n}\cup \{\infty\}$-valued stopping times $(\rho_n)_{n=1}^\infty$
 such that $\P(\varrho_n<\infty)<\eps$, and the  stopped processes $A^{n,\varrho_n}=(A^n_{\frac j{2^n}\wedge \rho_n})_{j=1}^{2^n}$, $M^{n,\varrho_n}=(M^n_{\frac j{2^n}\wedge \rho_n})_{j=1}^{2^n}$ satisfy 
\begin{align}
&TV(A^{n,\varrho_n})= \sum^{2^n (\varrho_n \wedge 1)}_{j=1} \left| A^n_{\frac{j}{2^n}} - A^n_{\frac{j-1}{2^n}}\right|&\leq C, \\
 &\|M_1^{n,\varrho_n}\|^2_\Ltwo= \| M^n_{\varrho_n \wedge 1}\|^2_\Ltwo &\leq C.
\end{align}
\end{proposition}

The proof of Theorem \ref{DiscreteApp} will be obtained as a consequence of three lemmas, the first of which is a slightly altered version of \cite[Lemma 7.4]{DeSc94}:
\begin{lemma}\label{l18} %(compare \cite[Lemma 7.4]{DeSc94})
%If the process as in \eqref{W18} above $S$ satisfies the condition of {\sffamily no free lunch with vanishing risk and small investment}, 
Under the assumptions of Proposition \ref{DiscreteApp}, the sequence of random
variables $(QV^n)^\i_{n=1}$ is bounded in $\Lzero$, where
$$QV^n=\sum^{2^n}_{j=1} \big(S_{\frac{j}{2^n}} -S_{\frac{j-1}{2^n}}\big)^2.$$
\end{lemma}
\begin{proof}  %Suppose not. Then there is some $\alpha>0$ and an increasing sequence $(n_k)_{k=1}^\infty$ such that 
%\begin{equation} \label{Eq1}
% \P[QV^{n_{k}} \geq 2 k +2] \geq \alpha. 
%\end{equation}
Set $h^n_t=- \sum^{2^{n}}_{j=1} 
S_{\frac{j-1}{2^n}} \I_{ ] \frac{j-1}{2^n}, \frac{j}{2^{n}}] }(t) $. 
Then $\| h^n \|_\LIOO \leq 1$ since $\| S \|_\infty \leq 1.$
Moreover,
$$ (h^n \cdotp S)_{t}=\tfrac{1}{2} \sum_{j=1}^{2^n} \big(S_{\frac{j}{2^n}\wedge t}
-S_{\frac{j-1}{2^{n}}\wedge t}\big)^2 + \tfrac{1}{2}( S_{0}^2-S_{t}^2)\geq -\tfrac12. $$
For $t=1$ we find 
\begin{align}\label{almostQV}
(h^n \cdotp S)_{1}=\tfrac{1}{2} QV^n + \tfrac{1}{2}( S_{0}^2-S_{1}^2).\end{align}
Since $S$ satisfies no free lunch with vanishing risk and small investments 
the family $\{(h^n\cdot S)_1:n\geq 1\}$ is bounded in $\Lzero$, hence \eqref{almostQV} proves the lemma.
\end{proof}
For $c>0$ we define, for each $n\geq 1$,
\begin{align*}
\sigma_n(c)
=\inf \Big\{ \tfrac{k}{2^n} : \sum^k_{j=1} \big(S_{\frac{j}{2^n}}-S_{\frac{j-1}{2^n}}\big)^2 \geq c-4\Big\}. \end{align*}
The $\left\{\tfrac{1}{2^n},\ldots ,\tfrac{2^n-1}{2^n} ,1\right\} \cup \{+\i\}$-valued functions $\sigma_n(c)$ are stopping times with respect to the filtration $(\F_t)_{0\leq t \leq 1}.$
By the preceding lemma %, for $\ve >0$ 
there is a constant  $c_1 >0$ such that, for all $n\geq 1$,
\begin{align}\label{COneAss} \p\left[\sigma_n \left(c_1\right) <\i\right] <\tfrac\eps2.\end{align}
\begin{lemma}
%Under the assumptions of Lemma \ref{l18} and using \eqref{W13} and \eqref{W13a}, for fixed $C>0$, 
Under the assumptions of Proposition \ref{DiscreteApp} and assuming that $c_1$ satisfies \eqref{COneAss} %for all $n\geq 1$, 
the stopped martingales $M^{n,\sigma(c_1)}=(M^n_{\frac{j}{2^n}})_{j=0}^{2^n(\sigma_n 
(c_1) \wedge 1)}$ are bounded in $L^2(\Omega,\F,\p)$ by
$$\left\|M^{n,c_1}\right\|^2_{L^2 (\Omega,\F,\p)} \le c_1.$$
\end{lemma}
\begin{proof} Fix $n\in \N$.
%We use the abbreviations $\ovl S=S^{\sigma(c_1)}, \ovl M=M^{n,\sigma(c_1)}, \ovl A=A^{n,\sigma(c_1)}$.
%Fix $n \in \N.$ 
For any $k \in \{0,1, \ldots , 2^n-1 \}$ 
\begin{align*}
 \E [(S^{\sigma_n(c_1)}_{\frac{k}{2^n}} -S^{\sigma_n(c_1)}_{\frac{k-1}{2^n}})^2]&=
\E [( M^{n,\sigma_n(c_1)}_{\frac{k}{2^n}} - M^{n,\sigma_n(c_1)}_{\frac{k-1}{2^n}}+ A^{n,\sigma_n(c_1)}_{\frac{k}{2^n}} -A^{n,\sigma_n(c_1)}_{\frac{k-1}{2^n}})^2] \\ 
%&= \E [(M_{\frac{k}{2^n}} -M_{\frac{k-1}{2^n}})^2]
%+2 \E [(M_{\frac{k}{2^n}} -M_{\frac{k-1}{2^n}})(A_{\frac{k}{2^n}} -A_{\frac{k-1}{2^n}})] \\
%&+\E [(A_{\frac{k}{2^n}} -A_{\frac{k-1}{2^n}})^2] \\ &=
& =\E [( M^{n,\sigma_n(c_1)}_{\frac{k}{2^n}} -M^{n,\sigma_n(c_1)}_{\frac{k-1}{2^n}})^2]+\E [(A^{n,\sigma_n(c_1)}_{\frac{k}{2^n}} - A^{n,\sigma_n(c_1)}_{\frac{k-1}{2^n}})^2] \\ &\geq 
\E [( M^{n,\sigma_n(c_1)}_{\frac{k}{2^n}} - M^{n,\sigma_n(c_1)}_{\frac{k-1}{2^n}})^2].
\end{align*}
Hence, we obtain
\begin{align*} 
\E\big[(M^{n,\sigma_n(c_1)}_{1})^2]&=\E[ (M^{n,\sigma_n(c_1)}_{1})^2]-\E[ (M^{n,\sigma_n(c_1)}_{0})^2\big]\\
&=\E \Big[ \sum_{k=1}^{2^n} \left( M^{n,\sigma_n(c_1)}_{\frac{k}{2^n}} - M^{n,\sigma_n(c_1)}_{\frac{k-1}{2^n}}\right)^2 \Big]\leq c_1.\qedhere
\end{align*}
%\leq \\ \E [ \sum^{2^n}_{k=1} (S_{\frac{k}{2^n}} -S_{\frac{k-1}{2^n}})^2 ] \leq c_{1}.$
\end{proof}
We write $A^{n,\sigma_n(c_1)}$ for the stopped process $(A^n_{\frac{j}{2^n}})^{2^n (\sigma_n(c_1)\wedge 1)}_{j=0}$ %and by $TV^{n}$ the total variation random variable 
and  abbreviate
$$TV^{n}= TV(A^{n,\sigma_n(c_1)})  =\sum^{2^n (\sigma_n (c_1)\wedge 1)}_{j=1} \big| A^n_{\frac{j}{2^n}} - A^n_{\frac{j-1}{2^n}}\big|.$$
\begin{lemma}
%Under the assumptions of Lemma \ref{l18}, for fixed $C>0$, 
Under the assumptions of Proposition \ref{DiscreteApp},
the sequence $(TV^{n})^\i_{n=1}$ is bounded in $\Lzero.$
\end{lemma}
\begin{proof} Suppose not. Then there is some $\alpha>0$ and for each $k$ some $n$ such that 
\begin{equation} \label{Eq1}
 \P[TV^{n} \geq  k] \geq \alpha. 
\end{equation}
For $n\geq 1$ let $b^n_{ {j-1}}=\mbox{sign}\big(A^{n,\sigma_n(c_1)}_{ \frac{j}{{2^{n}}}}-A^{n,\sigma_n(c_1)}_{ \frac{j-1}{{2^{n}}}}\big)$ and define
$$h^n(t)=\sum^{2^{n}}_{j=1}
b^n_{ {j-1}}
\I_{ ] \frac{j-1}{{2^{n}}}, \frac{j}{2^{n}}] }(t).$$ 
 Note that $\| h^n(t) \|_\infty \leq 1.$
Also, $(h^{n,\sigma_n(c_1)} \cdotp S)_{t}=(h^{n} \cdotp S^{\sigma_n(c_1)})_{t} $ can be estimated from below by
\begin{align}
\sum_{ j \leq t2^n} 
b^n_{ {j-1}}\Big[A^{n,\sigma_n(c_1)}_{ \frac{j}{2^n}}-& A^{n,\sigma_n(c_1)}_{ \frac{j-1}{2^n}}+M^{n,\sigma_n(c_1)}_{ \frac{j}{2^n}}-M^{n,\sigma_n(c_1)}_{ \frac{j-1}{2^n}}\Big] \nonumber  
+ b^n_{\lfloor t2^n\rfloor  }
\Big[S^{\sigma_n(c_1)}_{ t}-S^{\sigma_n(c_1)}_{ \frac{\lfloor t2^n\rfloor  }{2^n}}\Big] \nonumber 
\\  &\quad  \geq \  \big(h^{n,\sigma_n(c_1)} \cdotp A\big)_{ \frac{\lfloor t2^n\rfloor  }{2^n}}+ \big(h^{n,\sigma_n(c_1)} \cdotp M\big)_{ \frac{\lfloor t2^n\rfloor  }{2^n}}-2.\label{TheMartingalePart}
 %\\ &=TV^{n}+(h^k \cdotp M)_{1}
 \end{align}
As $\|(h^{n,\sigma_n(c_1)} \cdotp M)\|^2_\Ltwo\leq\|M^{n,\sigma_n(c_1)}\|^2_\Ltwo\leq c_1$ we obtain in particular that
\begin{align}
\big(h^{n,\sigma_n(c_1)} \cdotp S\big)_{1}=TV^n+\big(h^{n,\sigma_n(c_1)} \cdotp M^n\big)_{1} \nonumber 
\end{align}
does not remain bounded in $\Lzero$.
%We note that the first term $(TV^{n_{k}})_{k=1}^{\infty}$ is not bounded in $L^0(\Omega,\F,\p).$ Second term 
%$((h^k \cdotp M)_{1})_{k=1}^{\infty}$ is bounded in $L^2(\Omega,\F,\p)$ (hence in $L^0(\Omega,\F,\p)$) by the computation shown below:

We would like to assure that \eqref{TheMartingalePart} is uniformly bounded from below, but since we don't have a proper control on the martingale part, we need to perform some further stopping.
%As $\|(h^{n,\sigma_n(c_1)} \cdotp M)\|^2_\Ltwo\leq\|M^{n,\sigma_n(c_1)}\|^2_\Ltwo\leq c_1$,
By Doob's maximal inequality 
$$ \E \Big[\sup_{1\leq j\leq 2^n} \big((h^{n,\sigma_n(c_1)} \cdotp M)_j\big)^2\Big] \leq 4 c_1.$$
Hence for $c_2$ sufficiently large 
$$\tau_n=\inf\{\tfrac j{2^n}: |(h^{n,\sigma_n(c_1)} \cdotp M)_{\frac j{2^n}}|\geq c_2\}$$
satisfies $\P[\tau_n<\infty]= \tfrac\alpha 2$.
We thus obtain that $(h^{n,\tau_n\wedge\sigma_n(c_1)} \cdotp S)_{t}, n\geq 1$ is uniformly bounded from below, whereas 
$(h^{n,\tau_n\wedge\sigma_n(c_1)} \cdotp S)_{1}\geq 1$ is still unbounded in $\Lzero$, hence we obtain a free lunch with vanishing risk and small investments.
 \end{proof}
 \begin{proof}[Proof of Proposition \ref{DiscreteApp}.] We define $\tau_n(c) = \inf\{\tfrac{k}{2^n} :\sum^k_{j=1} |A^n_{\frac{j}{2^n}} - 
A^n_{\frac{j-1}{2^n}}|\geq c-2\},$ so that the stopped processes
$ A^{n,\tau_n (c)}, n\geq 1$ are bounded in total variation by $c$.
%\end{proof}
By the preceding lemma %, for $\ve >0$ 
there is a constant  $c_2 >0$ such that, for all $n\geq 1$,
$$ \p\left[\tau_n \left(c_2\right) <\i\right] <\tfrac\eps2.$$
Finally set
$\varrho_n  =\sigma_n (c_1) \wedge \tau_n (c_2)$ and  $C= c_1 \vee c_2$. 
%This completes the proof of Proposition \ref{DiscreteApp}.
\end{proof}

In the next step we extend the decompositions obtained in Proposition \ref{DiscreteApp} to continuous time. % and show that the stopping can be performed in a ``uniform way''. 
In the course of the proof we will use the following technical but elementary lemma.

\begin{lemma}\label{VariationBoundTrick}
Let $f,g:[0,1]\to \RR$ be deterministic functions such that $f$ takes only finitely many values and is left-continuous, i.e.\ $f$ can be written in the form 
\begin{align}
f= \sum_{k=1}^N f(s_{k})\I_{]s_{k-1},s_k]}
\end{align}
for appropriate $0\leq s_0\leq \ldots\leq s_N\leq 1$.
 For $t\in [0,1]$ define (in analogy to  \eqref{W3})
\begin{align}
(f\cdot g)_t= \sum_{k= 1}^{n(t)}  f(s_k)\big(g(s_k)-g(s_{k-1})\big)+ f(s_{n(t)})\big(g(t)-g(s_{n(t)})\big), 
\end{align}
where $n(t)\leq N$ is the maximal number subject to the condition  $s_{n(t)}<t$.
  For any increasing finite sequence $0\leq t_0\leq \ldots \leq t_m\leq 1$ we then have
\begin{align}\notag
\sum_{i=1}^m|(f\cdot g)(t_i)-(f\cdot g)(t_{i-1})|
\ \leq \ 2\ TV(f) \cdot  \|g\|_\infty + \|f\|_\infty\cdot  \sum_{i=1}^m |g(t_i)-g(t_{i-1})|.
\end{align}
\end{lemma}
\begin{proof}
Define for $i\in\{0,\ldots, m-1\} $ numbers $t_{i,0}\leq t_{i,1} \leq \ldots \leq t_{i,n} $ satifying $t_i=t_{i,0} $ and $t_{i+1}= t_{i,n}$ and so that all jumps of $f$ on the interval $[t_0,t_m]$ occur at some $t_{i,j}$. Then we obtain
\begin{align*}
&\sum_{i=1}^m|(f\cdot g)(t_i)-(f\cdot g)(t_{i-1})|%=\sum_{i=1}^m\Big| \sum_{j=1}^n (f\cdot g)(t_{i,j})-(f\cdot g)(t_{i,j-1})\Big|
= \sum_{i=1}^m\Big| \sum_{j=1}^n f(t_{i,j-1}) \big( g(t_{i,j})- g(t_{i,j-1})\big)\Big|\\
=&  \sum_{i=1}^m\Big|\Big(\sum_{j=1}^n \big(f(t_{i,j}) - f(t_{i,j-1}) \big) \big( g(t_{i,n})- g(t_{i,j-1})\big)\Big) +f(t_{i,0})\big( g(t_{i,n})- g(t_{i,0})\big) \Big|\\
\leq &\ 2\  TV(f) \cdot  \|g\|_\infty + \|f\|_\infty\cdot  \sum_{i=1}^m |g(t_i)-g(t_{i-1})|.\qedhere
\end{align*}
\end{proof}

\begin{proposition}\label{ContinuousApp}
Let $S=(S_t)_{\1}$ be a c\`{a}dl\`{a}g, adapted process satisfying $S_0=0$, $\| S\|_\LIOO\leq 1$ and the condition of    {\sffamily no free lunch with vanishing risk and little investment} (Definition \ref{NFLVRLI}).

For $\eps>0$ there exist a constant $C>0$, a $[0,1]\cup\{\infty\}$-valued stopping time $\alpha$ such that $\P[\alpha<\infty]<\eps$ and a sequence 
of continuous time c\`adl\`ag, adapted processes $\mathcal A^n, \mathcal M^n$ such that $\mathcal A(0)=\mathcal M(0)=0$, $(\mathcal M^n)$ is a martingale and
\begin{align}
 \mathcal A^{n,\alpha} +\mathcal M^{n,\alpha}&= S^{\alpha},\label{DecompositionPart}\\
 \|\mathcal M^{n,\alpha}\|_\Ltwo^2&\leq C, \label{MUBounded}  \\
 \sum_{j=1}^{2^n}|\mathcal  A^{n,\alpha}_{\frac j{2^n}}-\mathcal  A^{n,\alpha}_{\frac {j-1}{2^n}}|&\leq C.\label{AUBounded}
\end{align}
\end{proposition}
\begin{proof} Fix $\eps>0$ and let $C,M^n, A^n,\rho_n$ be as in Proposition \ref{DiscreteApp}. As a first step we extend the discrete time martingales
%$(M^n_{\frac{j}{2^n}})^{2^n}_{j=0}$ 
to martingales in continuous time (which, by slight abuse of notation, we still denote by $(M^n_t)_{\1}$) by letting 
$$M^n_t =\E[M^n_1 | \F_t], \qquad \1.$$
We also extend the discrete processes $(A^n_{\frac{j}{2^n}})^{2^n}_{j=0}$ to processes $(A^n_t)_{\1}$ by letting
$$ A^n_t= S_t -M^n_t, \qquad \1.$$
We note that these extended processes $(A^n_t)_{\1}$ need neither be predictable nor do we have a control on their total variation.
But we do have a control on the total variation of the restriction of the stopped process $(A^{n,\rho_n}_t)_{\1}$ to the grid $\{0, \tfrac{1}{2^n},\ldots , 
\tfrac{2^n -1}{2^n}, 1\}$, 
i.e.\ 
\begin{align}
\label{FirstConVarBound}
\sum^{2^n (\varrho_n \wedge 1)}_{j=1} \left| A^n_{\frac{j}{2^n}} - A^n_{\frac{j-1}{2^n}}\right|&\leq C.
\end{align}
We also note for further use that for $j\in \{0, \tfrac{1}{2^n},\ldots , 
\tfrac{2^n -1}{2^n}, 1\}$ and $t\in ]\tfrac{j-1}{2^n},\tfrac{j}{2^n}]$
\begin{align}\label{IntermediateBound}
\|A^n_t-A^n_{\frac{j}{2^n}}\|_\LI \leq 2
\end{align}
which readily follows from the representation
\begin{align}
\notag{}A^{n}_t&= S_t -M^{n}_t =S_t-\E[M^n_{\frac{j}{2^n}}|\F_t]\\
\notag{}&=S_t-\E[S_{\frac{j}{2^n}}-A^n_{\frac{j}{2^n}}|\F_t]=A^n_{\frac{j}{2^n}}- \big(S_t-\E[S_{\frac{j}{2^n}}|\F_t]\big).
\end{align}
Combining \eqref{FirstConVarBound} and \eqref{IntermediateBound} (and using that $A^n$ is c\`adl\`ag) we find that
\begin{align}\label{AUniformBound}
\| A^{n,\rho_n}\|_\LIOO \leq C+2.
\end{align}
The most delicate issue in the present proof is to pass --- in some sense --- to a limit of the stopping times $\rho_n$ in order to find the desired stopping time $\alpha.$
To this end, we define the left continuous process $R^n=\I_{
\llbracket
0,\varrho_n \wedge 1
\rrbracket
}$ as
\begin{align*}
R^n_t=\left\{
\begin{array}{cl}
1,&\mbox{ for }  0\le t\le\varrho_n ,\\
0,&\mbox{ for }  \varrho_n < t \le 1,\\
\end{array}\right.
\end{align*}
which is a decreasing, simple predictable integrand starting at $R^n_0 =1$ and satisfying
$\E\left[R^n_1\right] \geq 1-\ve.$
Also note that $A^{n,\rho_n}=(R^n\cdot A^n)$ and $M^{n,\rho_n}=(R^n\cdot M^n)$.

\medskip

We now apply Koml\'os' Lemma \ref{LtwoKomlos} to the sequence  $(R^n_1)_{n=1}^\infty$ of random variables in $\Linfty$ to pick, for each $n\geq 1$, convex combinations   
$\R_1^n =\sum^{N_n}_{i=n} \mu_i^n R_1^i$ and a random variable $\R_1\in \Linfty$ such that 
$\lim_{n\to \infty}\R^n_1=\R_1,$ convergence taking place almost surely.
Note that by dominated convergence
%\begin{align*}
$\E \left[\R_1\right] \geq 1-\ve.$ 
%\end{align*}
Subsequently we also consider the convex combinations 
%\begin{equation}\label{W22}
$\R^n =\sum^{N_n}_{i=n} \mu_i^n R^i$ 
%\end{equation}
of the processes $(R^n)_{n=1}^\infty$, and note that convergences of $\R^n_t$ is of course only granted if $t=1.$

In order to analyze the sequence $(\R^n\cdot S)^\i_{n=1}$ of simple integrals %\eqref{W22}
 we write 
\begin{align}\nonumber
\R^n\cdot S &= \Big(\sum^{N_n}_{i=n} \mu_i^{n} R^i\Big) \cdot S  =\sum^{N_n}_{i=n} \mu^{n}_i \Big( R^i\cdot (M^i +A^i)\Big) \\
& =\sum^{N_n}_{i=n} \mu_i^n (R^i \cdot M^i) +\sum^{N_n}_{i=n} \mu_i^n (R^i \cdot A^i). \label{W25}
\end{align}
Note that for each $n\in \N$, the first term is a martingale bounded in $\|.\|_2$ by $C^{\frac12}$, while the total variation of the  second term on the grid $\{0,\tfrac{1}{2^n},\ldots,1\}$ is bounded by $C$.

%Note that, for $n\in\N$, the first term is a martingale bounded in $L^2$ by 
%\begin{equation}\label{W25a}
%\Big\|\sum^{N_n}_{i=n} \mu_i^n (R^i \cdot M^i)\Big\|_{L^2 (\p)} \le 
%\sum^{N_n}_{i=n} \mu_i^n \left\|R^i\cdot M^i\right\|_{L^2(\p)} \le C^{\frac12}.
%\end{equation}
%The second term in \eqref{W25} is an adapted, c\`adl\`ag process which has bounded variation on the grid $\{0,\tfrac{1}{2^n},\ldots, \tfrac{2^n-1}{2^n},1\}$. Indeed  by \eqref{FirstConVarBound} we have a.s.\
%\begin{align}\label{W25b}
%& \sum^{2^n}_{j=1} \Big| \Big(\sum^{N_n}_{i=n} \mu_i^n (R^i \cdot A^i) \Big)_{\frac{j}{2^n}} - 
%\Big( \sum^{N_n}_{i=n}\mu_i^n (R^i\cdot A^i) \Big)_{\frac{j-1}{2^n}}\Big| \leq C %\\
%\leq\  & \sum^{N_n}_{i=n} \mu_i^n \left(\sum^{2^i}_{j=1} \left| (R^i \cdot A^i)_{\frac{j}{2^i}} - (R^i\cdot A^i)_{\frac{j-1}{2^i}} \right| \right) \le C.
%\end{align}
Define the stopping time $\alpha_n$ by 
$$ \alpha_n = \inf\left\{t: \R^n_t < \tfrac12\right\}.$$
As $\E\left[\R^n_1\right] \geq 1-\ve,$ we deduce from the inequality
$\ve \geq\E \left[1-\R^n_1\right] \geq \tfrac{1}{2} \p \left[\alpha_n <\i\right]$
that
$\p \left[\alpha_n <\i\right] \le 2\ve.$
Define $\S^n$ by
$$\S^n =(\R^n)^{-1} \I_{\llbracket 0, \alpha_n \rrbracket},$$
so that $\S^n$ is a simple predictable integrand satisfying $\| \S^n \|_\LIOO \le 2.$
%\begin{equation}\label{W26}
%\| \S^n \|_\LIOO \le 2.
%\end{equation}
Note that $\S^n \cdot (\R^n\cdot S) = (\S^n \R^n)\cdot S=\I_{\llbracket 0,\alpha_n \rrbracket }\cdot S$ equals the stopped process $S^{\alpha_n}$ defined by%\marginpar{In fact it does not equal the stopped process since multiplying means that we set everything to $0$ after $\alpha_n$. So we stop after $\alpha_n$. I need to check if that has bad consequences.} 
\begin{align*}
S^{\alpha_n}_t=\left\{
\begin{array}{cl}
S_t,&\mbox{ for }  0\le t\le \alpha_n \wedge 1,\\
S_{\alpha_n},&\mbox{ for }  \alpha_n\le t\le 1.\\
\end{array}\right. 
\end{align*} 
We therefore end up with the following extension of \eqref{W25}
\begin{align}\label{W27}
S^{\alpha_n} =\S^n \cdot (\R^n \cdot S) &=\underbrace{\S^n\cdot \Big( \sum^{N_n}_{i=n} \mu_i (R^i\cdot M^i)\Big)}_{=:\mathcal M^n} +\underbrace{\S^n\cdot \Big( \sum^
{N_n}_{i=n}  \mu_i (R^i\cdot A^i)\Big)}_{=:\mathcal A^n} %\\ \nonumber &=: {\M}^n + \mathcal{A}^n.
\end{align}
Next we establish that  $\mathcal{M}^n$ and $\mathcal{A}^n$  are bounded  as required. %stated in \eqref{MUBounded} and \eqref{AUBounded}.
As $\| \S^n \|_\LIOO \le 2$, %\eqref{W26} and \eqref{W25a} yield
\begin{equation}\label{W27a}
\| \M^n\|_{L^2 (\p)} \le 2 C^{\frac12}.
\end{equation} 
%To see that $\mathcal{A}^n$ is bounded we apply 
Applying Lemma \ref{VariationBoundTrick} %(for each $\omega\in\Omega$) 
to the functions
$\S^n_t $ (which a.s.\ satisfy $TV(\S^n_t )\leq 3, \|\S^n_t\|_\infty\leq 2$) and $\sum^
{N_n}_{i=n}  \mu_i (R^i\cdot A^i)_t$ (whose total variation
on  $\{0,\tfrac{1}{2^n},\ldots,1\}$ is bounded by $C$
%
%satisfy \eqref{W25b} and 
and which are
are %a.s.\ 
uniformly bounded by \eqref{AUniformBound}) we obtain %a.s. 
\begin{equation}\label{W28}
\sum^{2^n}_{j=1} \left| \mathcal{A}^n_{\frac{j}{2^n}} - \mathcal{A}^n_{\frac{j-1}{2^n}} \right|\le  
6\cdot (C+2)+2 \cdot C.
\end{equation}
We thus have established the boundedness  results \eqref{MUBounded} and \eqref{AUBounded} claimed in Proposition \ref{ContinuousApp}, except for the fact that the stopping times $\alpha_n$ still depend on $n$. 

\medskip

 We claim that there exists 
 an increasing sequence 
$(n_k)^\i_{k=1}$ such that the stopping time
$\alpha =\inf_{k\geq 1} \alpha_{n_k}$
satisfies
\begin{equation}\label{W32}
\p \left[\alpha <\i\right] \le 4\ve.
\end{equation}
Combining $\E [\R_1]\geq 1-\eps$ with the inequality $(1-a) \p \left[\R_1 \le a \right] \le  \ve$, we have 
$\p \left[\R_1 \le \tfrac23 \right] \le 3 \ve$
and we know that the sequence of random variables $(\R^n_1)^\i_{n=1}$ converges a.s.\  to $\R_1.$ % But then also $ \R^n_1\to \R_1$ in probability.
Hence there is an increasing sequence $(n_k)^\i_{k=1}$ such that for all $k \geq 1$ 
$$\P(|\R_1^{n_k}-\R_1|\geq \tfrac1{15})\leq \eps 2^{-k}.$$
It follows that 
$\p \left[\inf_{k\geq 1} \R^{n_k}_1  \le\tfrac35 \right]% \leq \p \left[\R_1 \le \tfrac23 \right] + \sum_{k=1}^\infty \P(|\R_1^{n_k}-\R_1|\geq \tfrac1{15}) 
 \le 4\ve,$
which implies \eqref{W32}.

\medskip

Summing up we obtain that the sequences $(\mathcal M^{n_k})^\infty_{k=1}$ and $(\mathcal A^{n_k})^\infty_{k=1}$ satisfy Proposition \ref{ContinuousApp}.
\end{proof}

\section{Proof of the main Theorems}

The major work has been done in Proposition \ref{ContinuousApp}; it is now sufficient to ``pass to a limit'' to establish Theorem \ref{BMT}.

\begin{proof}[Proof of Theorem \ref{BMT}.]
%For a given locally bounded right continuous, adapted process $S=(S_t)_{\1}$ and $\ve >0$ we may find a stopping time $\tau$ such that  $\P[\tau <\i] <\ve$ and # process $S^\tau$ is uniformly bounded.

%We write $S$ for the stopped process and assume w.l.g.~that $\|S\|_\i \le 1.$ Going with $S$ through the procedure of the present section II we either arrive at a {\it free lunch with vanishing risk and small investment}, or we eventually find by Lemma \ref{l27}  and \ref{l31} asubsequence $(n_k)^\i_{k=1}$ (still denoted by $n$), a $[0,1] \cup \{\i\}$-valued stopping time $\alpha$ such that $\p[\alpha <\i] <4 \ve,$
By stopping $S$, if necessary,  we may assume that $|S| $ is uniformly bounded by $1$. We fix $\eps>0$ 
and pick $C$,  $\alpha $ and, for each $n\geq 1$, $\M^{n}$ and $\mathcal{A}^{n},$ according to Proposition \ref{ContinuousApp}.
Denote by $\D$ the dyadic numbers in the interval $[0,1]$. We now apply Koml\'os Lemma (cf.\ the discussion in the Appendix) to the sequence of $\Ltwo$-martingales $(\M^{n,\alpha})_{n=1}^\infty$ and, for each $t\in \D$, to the sequence of bounded random variables $(\mathcal{A}_t^{n,\alpha})_{n=1}^\infty $ to find a c\`{a}dl\`{a}g martingale $\M$, a process $(\mathcal A_t)_{t\in \D}$ and for each $k$ some convex weights $\lambda_n^{n}, \ldots, \lambda_{N_n}^{n}$  such that 
\begin{align}
&\label{ww23} \lambda_n^{n} \M_1^{n, \alpha} + \ldots+\lambda_{N_n}^{n}  \M_1^{{N_n}, \alpha}\to \M\ &\mbox{ and}\\
&\label{ww24} \lambda_n^{n} \mathcal{A}_t^{n, \alpha} + \ldots+\lambda_{N_n}^{n}  \mathcal{A}_t^{{N_n}, \alpha}\to \mathcal A_t\ &\mbox{ for each $t\in \D$,} 
\end{align}
where the convergence in \eqref{ww23} and \eqref{ww24} is a.s.\ as well as in  $\Ltwo$.

For the process
$(\ma_t)_{t\in \mathcal{D}},$ indexed by the dyadic numbers $\mathcal{D}\subseteq [0,1],$ we then have 
\begin{equation}\label{W35}
\sum^N_{j=1} | \ma_{t_j} - \ma_{t_{j-1}} | \le C, \qquad \mbox{a.s.}
\end{equation}
for any collection $t_0 \le t_1\le \ldots \le t_N$ in $\mathcal{D}$.
Also, for every $t\in\mathcal{D}$ we have
$$S^\alpha_t = \M_t +\ma_t,% \qquad t\in\mathcal{D},
$$
so that $(\ma_t)_{t\in\mathcal{D}}$ is c\`adl\`ag on $\mathcal{D}$. Using \eqref{W35} we conclude that we may extend $(\ma_t)_{t\in\mathcal{D}}$
by right continuity to a process $(\ma_t)_{\1}$ via
$$\ma_t=\lim_{s\downarrow t,s\in\mathcal{D}} \ma_s, \qquad \1,$$
where a.s.\ the above limit exists for all $t\in[0,1].$ Using again right continuity we conclude that
$$S^\alpha_t = \M_t +\ma_t, \qquad \mbox{for} \ \1,$$
hence we obtain desired decomposition on $\llbracket 0,\alpha\wedge 1\rrbracket$. As $\P(\alpha<\infty)<\eps$ and $\eps>0$ was chosen arbitrarily, it follows that $S$ is a semi-martingale.
%
%\medskip
%
%The fact that every locally bounded semi-martingale $S$ has a decomposition $S=M+A$ into a local martingale $M$ and a predictable finite variation process $A$ can be seen, for instance, as a consequence of the Doob-Meyer Theorem. For the sake of completeness we  provide a self contained proof in the appendix.
\end{proof}

The unbounded case can be reduced from Theorem \ref{BMT} rather directly:
\begin{proof}[Proof of Theorem \ref{UBMT}.] %\marginpar{Warning: Here I haven't yet incorporated Walter's corrections} 
It is sufficient  to establish that $(1)  \Rightarrow  (2).$
 We collect the big jumps of $S$ in the process 
$$J_t=\sum_{0< s\leq t} \Delta S_s\I_{\{|\Delta{S_s}|\geq 1\}},$$
where $\Delta S_t=S_t-S_{t-}$. As $S_t$ is c\`{a}dl\`{a}g, $J$ is of finite total variation. It remains to prove that the locally bounded, c\`adl\`ag  process $X=S-J$ is a semi-martingale. We want to apply Theorem \ref{BMT}.  Let $(H^n)$ be a sequence of simple integrands satisfying %\eqref{SI}  and \eqref{VR} with respect to the process $X$. 
$\lim_{n}\big\| (H^n\cdot X)^{-} \big\|_\LIOT=\lim_{n} \| H^n \|_\LIOT = 0$. 
We claim that $(H^n)$ satisfies
\begin{align}\label{TheFirstOneHolds}
\lim_n \sup_{0\leq t\leq T} \left( (H^n\cdot S)_t\right)^{-}=0\quad \mbox{in probability}.
\end{align}
%\eqref{UR} with respect to the process $S$. 
Indeed
\begin{align}
\sup_{0\leq t\leq T} \left( (H^n\cdot S)_t\right)^{-} 
& \leq 
\sup_{0\leq t\leq T} \left( (H^n\cdot X)_t\right)^{-} +\sup_{0\leq t\leq T} \left| (H^n\cdot J)_t\right|\\
&\leq \sup_{0\leq t\leq T} \left( (H^n\cdot X)_t\right)^{-} + (|H^n| \cdot TV(J))_T. \label{thfsd}
\end{align}
In \eqref{thfsd} the first term tends to $0$ in probability since it tends to $0$ in $\|.\|_\LI$ and the second term tends to $0$ since $J$ is a finite variation process so that $TV(J)_T=\sum_{0< s\leq t} |\Delta S_s|\I_{\{|\Delta{S_s}|\geq 1\}}< \infty$ almost surely. %Thus we have established the claim. 

%Since we assume that $1.$ of Theorem \ref{UBMT} is valid, 
Having \eqref{TheFirstOneHolds} established, assumption (1) of Theorem \ref{UBMT} implies that $(H^n\cdot S)_T\to 0$ in probability. Since we have $(H^n\cdot J)_T\to 0$ this yields that also $(H^n\cdot X)_T=(H^n\cdot S)_T- (H^n\cdot J)_T$ converges to $0$ in probability. 

Thus $X$ satisfies {\sffamily no free lunch with vanishing risk and little investment} and hence is a semi-martingale by Theorem \ref{BMT}. 
\end{proof}

\section*{Appendix: Koml\'os' Lemma}

Koml\'os' orginal result reads as follows.%Koml\'os ().
\begin{lemma}\cite{Koml67}
\label{komlos} Let $(f_{n})_{n\geq1}$ be a sequence of measurable
 functions on a probability space $(\Omega, \F,\P)$ such that $\sup_{n\geq 1} \| f_n\|_1<\infty$.
 Then there exists a subsequence $(\tilde f_n)_{n\geq 1}$ such that the functions $
 \tfrac1n(\tilde f_1+\ldots+\tilde f_n)$ %, \quad n\geq 1$
 converge almost surely.
\end{lemma}

For our purposes a (much simpler) $L^2$-version is sufficient. For the convenience of the reader, we state it together with the short proof. 

\begin{lemma}\label{LtwoKomlos}
\label{komlos} Let $(f_{n})_{n\geq1}$ be a sequence of measurable
 functions on a probability space $(\Omega, \F,\P)$ such that $\sup_{n\geq 1} \| f_n\|_2<\infty$.
Then there exist functions  $g_n\in\operatorname*{conv}(f_{n},f_{n+1},\dots)$ such that $(g_n)_{n\geq 1}$ converges in $\|.\|_\Ltwo$ and almost surely.\footnote{We may also formulate this result in terms of Cesaro means. (This is due to \cite{BaSa30}, see also \cite[page 80]{FrBe90}.) But we don't need this.}
\end{lemma}

\begin{proof}
Let $\mathbb H$ be the Hilbert space generated by $(f_{n})_{n\geq1}$. 
For $n\geq 1$, denote by $K_n$ the (strong) closure of $ \operatorname*{conv}(f_{n},f_{n+1},\dots)$ which of course coincides with the weak closure by convexity. As the $K_n$ are weakly compact, we may pick $g\in \bigcap_{n=1}^\infty K_n$ and for each $n$ some $g_n\in \operatorname*{conv}(f_{n},f_{n+1},\dots)$ such that $g_n\to g$ in $\Ltwo$. By passing to a subsequence if necessary, we may assume that $g_n$ converges also almost surely.
\end{proof}

In the course of the paper we need to apply Lemma \ref{LtwoKomlos} to countably many sequences simultaneously. Just as we may extract a diagonal subsequence of a sequence of refining subsequences, we may do so analogously in the case of convex combinations.
Assume that for each $m\geq 1$, $(f_n^{m})_{n\geq 1} $ is a sequence of functions bounded in $\Ltwo$. Then we may choose for each $n$ some convex weights $\lambda_n^{n}, \ldots, \lambda_{N_n}^{n}$ (independent of $m$) such that 
$$\big( \lambda_n^{n} f_n^{m} + \ldots+\lambda_{N_n}^{n} f_{N_n}^{m}\big)_{n\geq 1}$$
converges for every $m\in\N$, in $\Ltwo$ and almost surely. 

To see this, one first uses  Lemma \ref{LtwoKomlos} to find convex weights $\lambda_n^{n}, \ldots, \lambda_{N_n}^{n}$ such that   
$( \lambda_n^{n} f_n^{1} + \ldots+\lambda_{N_n}^{n} f_{N_n}^{1})_{n\geq 1}$ converges. In the second step, one applies Lemma \ref{LtwoKomlos} to the sequence 
$( \lambda_n^{n} f_n^{2} + \ldots+\lambda_{N_n}^{n} f_{N_n}^{2})_{n\geq 1} $, to obtain convex weights which work for the first two sequences. Repeating this procedure inductively we obtain sequences of convex weights which work for the first $m$ sequences.  
Then a standard diagonalization argument proves the assertion.

\def\ocirc#1{\ifmmode\setbox0=\hbox{$#1$}\dimen0=\ht0 \advance\dimen0
  by1pt\rlap{\hbox to\wd0{\hss\raise\dimen0
  \hbox{\hskip.2em$\scriptscriptstyle\circ$}\hss}}#1\else {\accent"17 #1}\fi}

\end{document}